\documentclass[a4paper,11pt]{article}
\usepackage{latexsym}
\usepackage{amsmath}
\usepackage{amsthm}
\usepackage[pdftex]{graphicx}
\usepackage{float}
\usepackage{mathrsfs}
\usepackage{hyperref}
\usepackage{amssymb,amsfonts}
\usepackage{color}

\newtheorem{thm}{Theorem}[section]

\newtheorem{lem}[thm]{Lemma}

\def\R {{\mathbb R}}
\newcommand{\psd}[1]{{\R_\succeq^{#1 \vee #1}}}					

\begin{document}

%Author
\title{A note on the Lov\'{a}sz-Schrijver Semidefinite Programming Relaxation for Binary Integer Programs}
\author{Pietro Paparella\thanks{Department of Mathematics, Washington State University, Pullman, WA 99164-3113, USA (\href{mailto: ppaparella@math.wsu.edu}{\texttt{ppaparella@math.wsu.edu}}).}}
\date{}
\maketitle

%Abstract
\begin{abstract}
\noindent
Binary Integer Programming (BIP) problems are of interest due in part to the difficulty they pose and because of their various applications, including those in graph theory, combinatorial optimization and network optimization. In this note, we explicitly state the Lov\'{a}sz-Schrijver Semidefinite Programming (SDP) relaxation (in primal-standard form) for a BIP problem, a relaxation that yields a tighter upper-bound than the canonical Linear Programming relaxation. 

\end{abstract}

%Notation
\section{Notation}
In this note, the following notational conventions are adopted: 
\begin{enumerate}
\item $\R^{1+n}:=\left\{ \begin{bmatrix} x_0 \\ x \end{bmatrix} : x_0 \in \R, x \in \R^n \right\}$ and $\{ e_i \}_{i=0}^n$ denotes the canonical basis. 
\item The space of real $n \times n$ matrices is denoted by $\R^{n \times n}$. The space of real, symmetric $n \times n$ matrices is denoted by $\R^{n \vee n}$. The space of real, symmetric, positive definite (positive semidefinite) $n \times n$ matrices is denoted by $\R_\succ^{n \vee n}$ $\left( \R_\succeq^{n \vee n}\right)$.
\item The $(i,j)$ entry of a matrix $X$ is denoted by $x_{ij}$.
\item Positive definiteness (or positive semidefiniteness) of a  matrix $X$ is denoted by $X \succ 0$ $(X \succeq 0)$. 
\item For $X$, $Y \in \R^{n \times n}$, $X \bullet Y$ denotes the (Frobenius) inner product of the matrices $X$ and $Y$, defined by trace$(X^T Y)$. 
\item For $X \in \R^{n \times n}$, vec$(X)$ denotes the column-wise vectorization of a matrix $X$.
\item For $S \subseteq \R^n$, conv$(S)$ denotes its convex hull.
\item Diag$(x)$ denotes the $n \times n$ diagonal matrix with the vector $x \in \R^n$ on its diagonal. For $X \in \R^{n \times n}$, diag$(X)$ is the column vector of the diagonal entries of $X$. For matrices $A_1 \in \R^{n_1 \times n_1}, \dots, A_r \in \R^{n_1 \times n_1}$, Diag$\left( A_1, \dots, A_r \right) \in \R^d$ denotes the block-diagonal matrix with matrices $A_1, \dots,A_r$ along its block-diagonal, where $d := \sum^r_{k=1} n_k \times \sum^r_{k=1} n_k$.
\end{enumerate}

%Lov\'{a}sz-Schrijver
\section{Lov\'{a}sz-Schrijver Lift-and-Project Method}
% Subsection Lifted Matrix Variable
\subsection{Lifted Matrix Variable}
Consider the binary (or 0-1) integer program 
\begin{align*}
\begin{array}{r r l l l}
\text{maximize}    	& c^T x					\\
\text{subject to} 	& a_i^T x & \leq & b_i & i=1,\dots,m	\\
                    		& x & \in & \{ 0 , 1 \}^n  
\end{array}
\tag{BIP} \label{BIP}
\end{align*}
and its Linear Programming (LP) relaxation
\begin{align*}
\begin{array}{r r l l l}
\text{maximize}    	& c^T x			\\
\text{subject to} 	& a_i^T x & \leq & b_i & i=1,\dots,m	\\
                    		& x & \in &  [0 , 1]^n 
\end{array}
\tag{LPR} \label{LP}
\end{align*}

Let $P$ be the polytope defined by $P := \{ x \in \R^n : Ax \leq b \}$ (assume $Ax \leq b$ includes the $m$ inequalities $a_i^T x \leq b_i$ and the trivial inequalities $0 \leq x \leq 1$). Let $P_I$ denote the convex hull of the 0-1 vectors belonging to $P$. Note that solving $\eqref{LP}$ provides an upper bound on $\eqref{BIP}$, however this solution may not be integral and far from the actual solution. Notice that the polytope $P$, obtained by relaxing the condition $x \in \{ 0, 1 \}^n$ to $x \in [ 0 ,1 ]^n$, is an approximation of $P_I$.

Lov\'{a}sz and Schrijver \cite{Lov-Shri 1991} devised a method that generates nonlinear ``cuts" that better approximate $P_I$ than $P$. Instead of working with $x \in \{0,1\}^n$ in $\eqref{BIP}$, Lov\'{a}sz and Schrijver considered the lifted matrix variable 
\[ X := \begin{bmatrix} 1 \\ x \end{bmatrix} \begin{bmatrix} 1 & x^T \end{bmatrix} = \begin{bmatrix} 1 & x^T \\ x & xx^T \end{bmatrix}. \]
Note that $X$ has the following properties:
\begin{enumerate}
\item $X \in \psd{n+1}$. Indeed, $X$ is a symmetric, rank-one matrix with spectrum $\sigma(X) = \{ 1 + x^Tx, 0 \}$. 
\item $Xe_0=\text{diag}(X)$, i.e., the first column of $X$ equals the diagonal of $X$. This follows from $x_{ii} = x^2_i = x_i$ and $x_i \in \{0,1\}$. Moreover, following the symmetry of $X$, we have $Xe_0 = \text{diag}(X) = X^T e_0$, i.e., the first column, first row and diagonal of $X$ are equal.
\end{enumerate}

% Subsection Nonlinear Cuts
\subsection{Nonlinear Cuts} 
Note that for $i=1,\dots,m$, $j=1,\dots,n$ the inequalities
\begin{align*} 
(b_i - a_i^T x) x_j &\geq 0 \tag{1} \label{1}	\\
(b_i - a_i^T x) (1 - x_j) &\geq 0 \tag{2} \label{2}
\end{align*}
are valid for $x \in P$. 

Let $u_i := \begin{bmatrix} b_i & -a_i^T \end{bmatrix}^T$. One can verify (c.f. \cite{Dash 2001}) that $\eqref{1}$ and $\eqref{2}$ are expressible in terms of $X$ as 
\begin{align*}
u_i e_j^T \bullet X &\geq 0 \tag{3} \label{3}	\\
u_i (e_0 - e_j)^T \bullet X &\geq 0 \tag{4} \label{4}
\end{align*}
Further, the condition $Xe_0 = \text{diag}(X)$ becomes 
\[ e_j (e_0 - e_j)^T \bullet X = 0. \tag{5} \label{5}\] 
Finally, we require $X_{00} = 1$ which is expressible as 
\[ e_0 e_0^T \bullet X = 1. \tag{6} \label{6} \]

Lov\'{a}sz and Schrijver then propose the {\it cones}
\[ M_+(P) := \{ X \in \psd{n+1}: \eqref{3} \text{--} \eqref{6} \} \]
and
\[ N_+(P) : = \{ x \in \R^n : \begin{bmatrix} 1 & x^T \end{bmatrix}^T = \text{diag}(X), X \in M_+ (P) \} \]
and establish
\begin{lem}[See Lemma 1.1 in  \cite{Lov-Shri 1991}]  
$P_I \subseteq N_+(P) \subseteq P$. 
\end{lem}

Following \hyperref[Lemma 2.1]{Lemma 2.1}, solving $\max \{ c^T x : x \in N_+ (P) \}$ produces a tighter upper-bound for $\eqref{BIP}$ than $\eqref{LP}$. 

%SDP
\section{Semidefinite Programming (SDP)}
An SDP problem in \textit{primal form} is given by
\begin{align*}
\begin{array}{r r l l l}
\text{minimize}    	& C \bullet X					\\
\text{subject to} 	& A_i \bullet X & = & b_i &  i=1,\dots,m 	\\
                    		& X & \succeq & 0  
\end{array}
\tag{SDPP} \label{SDPP}
\end{align*}
where $A_i \in \R^{n \vee n}$, $b_i \in \R^n$, $C \in \R^{n \vee n}$ are the problem data, and $X \in \R_\succeq^{n \vee n}$ is the variable. 

An SDP problem in \textit{dual form} is given by
\begin{align*}
\begin{array}{r r l}
\text{maximize} &  b^T y &		\\
\text{subject to} & \sum_{i=1}^m y_i A_i  & \preceq C 	
\end{array}
\tag{SDPD} \label{SDPD}
\end{align*}
where $y \in \R^m$ is the variable. 

Note that the linear programming problem
\begin{align*}
\begin{array}{r r l l l}
\text{maximize}    	& c^T x				\\
\text{subject to} 	& a_i^T x & = & b_i & i = 1,\dots,m	\\
                    		& x & \geq &  0 
\end{array}
\end{align*}
becomes an SDP problem in primal form by setting $C:= \text{Diag}(c)$, $A_i:= \text{Diag}(a_i)$ and $X:=\text{Diag}(x)$ so that Semidefinite Programming is a generalization of Linear Programming. 

SDP has applications in eigenvalue optimization, combinatorial optimization, and system and control theory; furthermore, there are several approximation methods for solving SDP's. (See \cite{Todd 2001} or \cite{Van-Boyd 1996} for more detailed discussions concerning SDP.) 

%SDP Relaxation
\section{SDP Relaxation} 
Before we state the SDP primal-form problem explicitly, we prove the following lemma.
\begin{lem}\label{Lemma 4.1}
If $A \in \R^{n \times n}$, $X \in \R^{n \vee n}$ and $A':= \frac{1}{2}(A + A^T)$, then $A\bullet X = A' \bullet X$.    
\end{lem}
\begin{proof}
Following properties of the {\it trace} and {\it transpose} operators,
\begin{align*}
A' \bullet X
= \text{tr}\left( \frac{1}{2}(A + A^T)^T X \right) &= \frac{1}{2}\text{tr}(A^T X) + \frac{1}{2}\text{tr}(A X)	\\
&= \frac{1}{2}\text{tr}(A^T X) + \frac{1}{2}\text{tr}(X A^T)	\\
&= \frac{1}{2}\text{tr}(A^T X) + \frac{1}{2}\text{tr}(A^T X) 	\\
&= \text{tr}(A^T X) = A \bullet X. \qedhere
\end{align*}
\end{proof}
Following \hyperref[Lemma 4.1]{Lemma 4.1}, constraints $\eqref{3}$-$\eqref{6}$ can be written in terms of symmetric matrices (a requirement for the canonical primal- and dual-form SDP problems). 

Let $C := e_0 \begin{bmatrix} 0 & c^T  \end{bmatrix}$. Dash (c.f. \cite{Dash 2001}) demonstrated that solving $\max \{ c^T x : x \in N_+ (P) \}$ is equivalent to solving the SDP (in non-canonical form)
\begin{align*}
\begin{array}{r r l l l}
\text{maximize} 	& C \bullet X   													\\
\text{subject to} 	& \frac{1}{2} \left[ u_i e_j^T + \left( u_i e_j^T \right)^T \right] \bullet X & \geq & 0			\\
                  		& \frac{1}{2} \left[ u_i (e_0 - e_j)^T + \left( u_i (e_0 - e_j)^T \right)^T \right] \bullet X & \geq & 0	\\
                 	 	& \frac{1}{2} \left[ e_j (e_0 - e_j)^T + \left( e_j (e_0 - e_j)^T \right)^T \right] \bullet X & = & 0	\\
                 	 	& e_0 e_0^T \bullet X & = & 1										\\
                  		&  X & \succeq  & 0.											\\
                  		& i & = & 1,\dots,m 											\\
                 		& j & = & 1,\dots,n 											
\end{array}
\end{align*}
which, after introducing $2mn$ {\it surplus-variables}, becomes 
\begin{align*}
\begin{array}{r r l l l l}
\text{maximize} & C \bullet X   															\\
\text{subject to} & \frac{1}{2} \left[ u_i e_j^T + \left( u_i e_j^T \right)^T \right] \bullet X & - & s_{ij} & = 0			\\
                  & \frac{1}{2} \left[ u_i (e_0 - e_j)^T + \left( u_i (e_0 - e_j)^T \right)^T \right] \bullet X & - & \bar{s}_{ij} & = 0	\\
                  & \frac{1}{2} \left[ e_j (e_0 - e_j)^T + \left( e_j (e_0 - e_j)^T \right)^T \right] \bullet X & & &= 0			\\
                  & e_0 e_0^T \bullet X & & & = 1												\\
                  &  X & & & \succeq  0													\\
                  & & & i & = & 1,\dots,m 													\\
                  & & & j & = & 1,\dots,n. 				
\end{array}
\end{align*}
For $i=1,\dots,m$, $j=1,\dots,n$, define
\begin{enumerate}
\item $\bar{n} := 2mn + n +1$

\item$S:=[s_{ij}] \in \R^{m \times n}$ 

\item $\bar{S}:=[\bar{s}_{ij}] \in \R^{m \times n}$ 

\item $\bar{C} := \begin{bmatrix} -C & 0 \\ 0 & 0 \end{bmatrix} \in \R^{\bar{n} \times \bar{n}}$ 

\item $\bar{X} := \begin{bmatrix} X \\ & \text{Diag(vec($S$))} \\ & & \text{Diag(vec($\bar{S}$))} \end{bmatrix}  \in \R^{\bar{n} \times \bar{n}}$

\item $A_{ij}:= \text{Diag} \left( \frac{1}{2} \left[ u_i e_j^T + \left( u_i e_j^T \right)^T \right], 0, \dots,\underbrace{-1}_{n+1+m(j-1)+ i},0, \dots, 0 \right)  \in \R^{\bar{n} \times \bar{n}}$.

\item $A_{ij}:= \text{Diag} \left( \frac{1}{2} \left[ u_i (e_0 - e_j)^T + \left( u_i (e_0 - e_j)^T \right)^T \right], 0, \dots,\underbrace{-1}_{n+1+mn+m(j-1)+ i},0, \dots, 0 \right)  \in \R^{\bar{n} \times \bar{n}}$.

\item $\tilde{A}_{ij}:= \text{Diag} \left( \frac{1}{2} \left[ e_j (e_0 - e_j)^T + \left( e_j (e_0 - e_j)^T \right)^T \right], 0_{2mn \times 2mn} \right) \in \R^{\bar{n} \times \bar{n}}$

\item $A:= \text{Diag} \left( e_0 e_0^T, 0_{2mn \times 2mn} \right) \in \R^{\bar{n} \times \bar{n}}$
\end{enumerate}
so that the primal-form SDP relaxation of $\eqref{BIP}$ is
\begin{align*}
\begin{array}{r r r l}
\text{minimize}   	& \bar{C} \bullet \bar{X}			\\
\text{subject to} 	& A_{ij} \bullet \bar{X} & = & 0 		\\
                            	& \bar{A}_{ij}\bullet \bar{X} & = & 0 	\\
                            	& \tilde{A}_{ij}\bullet \bar{X} & = & 0 	\\
                            	& A \bullet \bar{X} & = & 1		\\
		    	& \bar{X} & \succeq & 0			\\
		    	& i & = & 1,\dots,m 			\\
		    	& j & = & 1,\dots, n.
\end{array} \label{1}
\end{align*}

\newpage
%Bibliography

\end{document}